\documentclass[12pt]{amsart}
\usepackage{amsmath,enumitem}
\usepackage[english,  activeacute]{babel}
\usepackage[latin1]{inputenc}
\usepackage{amssymb}
\usepackage{amsthm}
\usepackage{graphics}
\usepackage{array}
\usepackage{pdflscape}
\usepackage{a4wide}
\usepackage{color, url}
\usepackage{float}

\theoremstyle{plain}
\newtheorem{theorem}{Theorem}
\newtheorem{corollary}[theorem]{Corollary}

\theoremstyle{definition}

\newtheorem{example}[theorem]{Example}

\newcommand{\E}{{\mathcal E}}
\newcommand{\N}{{\mathbb N}}
\newcommand{\M}{{\mathbb M}}
\newcommand{\Odd}{{\mathcal O}}
\newcommand{\Z}{{\mathbb Z}}

\newcommand{\T}{{\mathbb T}}
\newcommand{\Pa}{{\mathcal P}}
\newcommand{\Lo}{{\mathcal L}}
\def\li{\text{\rm Li}}

\title[Some Applications of $S$-restricted Set Partitions]{Some Applications of $S$-restricted Set Partitions}

\author{Be\'ata B\'enyi}
\address{\noindent Faculty of Water Sciences, National University of Public Service, Budapest, HUNGARY}
\email{beata.benyi@gmail.com}

\author{Jos\'e L. Ram\'{\i}rez}
\address{\noindent Departamento de Matem\'aticas, Universidad Nacional de Colombia, Bogot\'a,  COLOMBIA}
\email{jlramirezr@unal.edu.co}

\date{\today}
\subjclass[2010]{Primary 11B83; Secondary  11B73, 05A19, 05A15.}
\keywords{Combinatorial identities, generating functions, $S$-restricted Stirling numbers,  lonesum matrices,  poly-Bernoulli numbers.}

\begin{document}
\begin{abstract}
In the paper, the authors present several new relations and applications for the combinatorial sequence that counts the possible partitions of a finite set with the restriction that the size of each block is contained in a given set. One of the main applications is in the study of lonesum matrices. 
\end{abstract}

\maketitle

\section{Introduction}
Stirling numbers of the second kind is an important number sequence in mathematics. The combinatorics of these numbers is very rich. Since it counts the fundamental combinatorial object, set partitions of $n$ elements into $k$ non-empty blocks, it arises often in enumeration problems. The arithmetical properties of these numbers as well as the analytical properties of their generating function set them into the focus of many research. For instance, it also plays in umbral calculus a central role. 

Stirling numbers were generalized in different ways, combinatorially: restrictions on the size of the blocks lead to restricted Stirling numbers and associated numbers (cf. \cite{Comtet}); restrictions of the appearance of distinguished elements lead to $r$-Stirling numbers \cite{Broder}. Considering different weights $q$-analogues arose in the literature (cf. \cite{mansour-2015a}). Furthermore, generalizations were introduced that connect to other numbers, showing that Stirling numbers are special cases of a large family of interesting number sequences. Though some ideas were already mentioned in early works, as in the fundamental book of Comtet \cite{Comtet}, recently the topic started to win again interests and more research groups are concerned with it (\cite{bona-2016a}, \cite{choijy-2006b}, \cite{Eng}, \cite{KLM2}, \cite{KLM},  \cite{MS}, \cite{Moll}, \cite{Wakhare}). 

We adjoin this research investigating the case when the size of each block is contained in a given set $S$. We consider some applications of these generalization, studying the modifications of well-known formulas/identities of the classical Stirling numbers. We also study the number sequences that are strongly related, as for example Bell numbers, Fubini numbers, poly-Bernoulli numbers. We involve into our study a large scale of methods from elementary combinatorics to iterated integrals. We emphasize the combinatorial connections throughout the paper.

The outline of the paper is as follows. Section 2 contains some known and  new combinatorial identities for the $S$-restricted Stirling and Bell numbers. For example, we generalize the Dobinski's formula. We also show the relation with the Riordan group.  Section 3 contains an extension of the Fubini numbers and some new identities. Section 4 presents an interesting application of the $S$-restricted partitions to the lonesum matrices. Finally, Section 4 introduces a generalization of the poly-Bernoulli numbers. 

\section{$S$-restricted Set Partitions and Stirling Numbers}

A partition of a set $[n]:= \{1, 2, \dots , n\}$ is a collection of pairwise disjoint subsets, called \emph{blocks}, whose union is $[n]$. 
For a block $B$, we let $|B|$ denote its cardinality. We denote by ${n\brace k}_S$ the number of partitions of $[n]$ having exactly $k$ blocks, with the additional restriction that the size of each block is contained in the set $S\subseteq \Z^{+}$.  We call this sequence \emph{$S$-restricted Stirling numbers of the second kind}. If $S=\{k_1, k_2, \dots \}$, then it is known that the $S$-restricted Stirling numbers of the second kind are given by
$${n\brace k}_S=\sum_{\substack{c_1k_1+c_2k_2+\cdots =n, \\ c_1+c_2+\cdots = k}}\frac{n!}{c_1!c_2!\cdots (k_1!)^{c_1}(k_2!)^{c_2}\cdots}.$$

Note that this  sequence is a  particular case of partial Bell polynomials \cite{Comtet}. Therefore, its exponential generating function is
\begin{align}\label{gfun}
\sum_{n=k}^\infty { n \brace k }_S \frac{x^n}{n!}=\frac{1}{k!}\left(\sum_{i\geq 1} \frac{x^{k_i}}{k_i!}\right)^k.
\end{align}

For $S=\Z^+$ we recover the well-known  \emph{Stirling numbers of the second kind} ${n\brace k}$. For the restriction $S=\{1, 2,  \dots, m\}$    we obtain the  \emph{restricted Stirling numbers of the second kind}  ${n\brace k}_{\leq m}$. This sequence  gives the number of partitions  of $n$ elements  into $k$ non-empty blocks, with the additional restriction that none of the blocks contains more than  $m$ elements. By setting $S=\{m, m+1, \dots\}$ we obtain the  \emph{associated  Stirling numbers of the second kind}  ${n\brace k}_{\geq m}$.   In this case the restriction is  that every block contains at least $m$ elements. Several combinatorial properties of these sequences can be found in \cite{BB,bona-2016a, choijy-2003a, choijy-2006b, Comtet, KLM, KomRam2, MS, MJR, Moll, Wakhare}.  Let $\mathcal{E}$ (resp. $\mathcal{O}$) be the set of even positive integers (resp. odd positive integers), then for $S=\mathcal{E}$ (resp. $S=\mathcal{O}$) we obtain the number of partitions of $[n]$ into non-empty blocks of even size,  (resp. odd size) denoted by ${n\brace k}_{\mathcal{E}}$ (resp. ${n\brace k}_{\mathcal{O}}$).

Let us denote by $E_m(x)$ the partial sum
$$E_m(x):=\sum_{k=0}^m \frac{x^k}{k!}.$$
From the generating function given in Equation \eqref{gfun} we obtain the following particular cases
\begin{align*}
\sum_{n=k}^\infty { n \brace k } \frac{x^n}{n!}&=\frac{1}{k!}\left(e^x-1\right)^k,\\
\sum_{n=k}^\infty { n \brace k }_{\leq m} \frac{x^n}{n!}&=\frac{1}{k!}\left(E_m(x)-1\right)^k,\\
\sum_{n=k}^\infty { n \brace k }_{\geq m} \frac{x^n}{n!}&=\frac{1}{k!}\left(e^x-E_{m-1}(x)\right)^k,\\
\sum_{n=k}^\infty { n \brace k }_{\E}\frac{x^n}{n!}&=\frac{1}{k!}\left(\cosh x-1\right)^k,\\
\sum_{n=k}^\infty { n \brace k }_{\Odd} \frac{x^n}{n!}&=\frac{1}{k!}\left(\sinh x\right)^k.
\end{align*}

In this paper we will use the symbolic method \cite{Flajolet} for deriving exponential generating functions. Let $\mathcal{A}$ and $\mathcal{B}$ be combinatorial classes with exponential generating functions
$$
A(x)=\sum_{\alpha\in \mathcal{A}}\frac{x^{|\alpha|}}{|\alpha|!}=\sum_{n\geq 0}A_n\frac{x^n}{n!}\mbox{ and }
B(x)=\sum_{\beta\in \mathcal{B}}\frac{x^{|\beta|}}{|\beta|!}=\sum_{n\geq 0}B_n\frac{x^n}{n!}.$$

$A_n$ (resp. $B_n$) is the counting sequence for objects in the class $\mathcal{A}$ (resp.~$\mathcal{B}$) with size $n$. Let $\mathcal{X}$ be the atomic class with generating function $x$. According to the theory, if we have a construction  for a class of combinatorial objects that is built up of classical basic constructions, as SET, $k$-SET, SEQ etc.; we can turn this construction automatically into a generating function for the counting sequence using the translation rules.
$S$-restricted Stirling numbers of the second kind counts the partition of an $n$ element set into $k$ non-empty blocks such that each block has size contained in $S$. For the construction of such partitions we need two basic constructions: $k$-sets and $S$-sets. The class of a $k$-set formed from $\mathcal{B}$ is a $k$-sequence modulo the equivalence relation that two sequences are equivalent when the components of one are the permutation of the components of the other. The translation rule is
$\mathcal{A}=\mbox{SET}_k(\mathcal{B})\Rightarrow A(x)=\frac{B(x)^k}{k!}$. The class of $S$-sets ($S=\{k_1,k_2,\ldots\}$) formed from $\mathcal{B}$ is the union of $k_i$-sets formed from $\mathcal{B}$. The translation rule is $\mathcal{A}=\mbox{SET}_{S}(\mathcal{B})\Rightarrow A(x)=\sum_{i\geq 1}\frac{B(x)^{k_i}}{k_i!}$. For instance, if $S=\{2,4,5\}$, $A(x)=\frac{B(x)^2}{2!}+\frac{B(x)^4}{4!}+\frac{B(x)^5}{5!}$.

Our construction for the class $S$-restricted partitions into $k$ non-empty blocks, denoted by $\mathcal{P}_{k;S}$, is
$$\mathcal{P}_{k;S}=\mbox{SET}_k(\mbox{SET}_S(\mathcal{X})).$$ The translation rules turn this construction into the generating function \eqref{gfun}.

Recently, Wakhare \cite{Wakhare} proved the following recurrence relation
\begin{align}
{n+1 \brace k }_S =\sum_{s\in S} \binom{n}{s-1}{n-s+1 \brace k-1}_S.
\end{align}
\begin{proof}
Distinguish one element in a set of $n+1$ elements. Take a partition of this set into $k$ blocks such that each block has size in $S$. The distinguished element is in one particular block which has size $s$ with $s\in S$. The other elements of this block can be chosen in $\binom{n}{s-1}$ ways and the remaining $n+1-s$ elements form a partition into $k-1$ blocks such that the size of each block is contained in $S$; for that we have ${n-s+1 \brace k-1}_S$ possibilities.
\end{proof}
It is well-known that the Stirling numbers of the second kind can be defined as the connecting coefficients between the falling factorial $(x)_n$, with $(x)_n:=x(x-1)\cdots (x-n+1)$ if $n\geq 1$ and $(x)_0=1$, and the monomials $x^n$. More concretely,
$$x^n=\sum_{k=0}^n{n \brace k}(x)_k.$$
It is possible to generalize the above identity by means of the potential polynomials (\cite[Theorem B, pp. 141]{Comtet}). Let $f_{S,t}(x)$ be the function defined by
$$f_{S,t}(x):=(1+E_S(x))^t,$$
where
$$
E_S(x)=\sum_{i\geq 1} \frac{x^{k_i}}{k_i!}.
$$
Then
 \begin{align}\label{eqbb}
\left. \frac{d^n}{dx^n}f_{S,t}(x)\right |_{x=0}:=f_{S,t}^{(n)}(0)=\sum_{k=0}^n{n \brace k}_{S} (t)_k.
\end{align}
Therefore  the polynomial $f_{S,t}^{(n)}(0)$ can be interpreted as the number of all functions $g$ from $[n]$ to $[t]$, such that $|g^{(-1)}(i)|\in S$ for any $i\in[t]$.

If we denote by $B_{n, S}$ the total number of set partitions of $[n]$ with the restriction that the size of each block is contained in  $S\subseteq \Z^{+}$, then it is clear that
$$B_{n, S}=\sum_{k=0}^n{n \brace k}_S.$$
In particular, if $S=\Z^+$ we recover the Bell numbers $B_n$ (\cite{OEIS}, A000110).
The exponential generating function of the sequence  $B_{n,S}$ is
 \begin{align}\label{gfunbell}
 \sum_{n=0}^\infty B_{n, S}\frac{x^n}{n!}&=\exp\left(\sum_{i\geq 1} \frac{x^{k_i}}{k_i!}\right).
 \end{align}
\begin{proof}
Here we can use again the symbolic method to derive this formula direct from the combinatorial interpretation of $B_{n,S}$. We denote by $\mathcal{P}_{S}$ the class of all partitions such that each block has size contained in $S$. For the construction we will need another construction: the $\mbox{SET}$ construction. A set formed from $\mathcal{B}$ is defined as the union of $k$-sets formed from $\mathcal{B}$: $$\mbox{SET}(\mathcal{B})=\{\epsilon\}\cup \mathcal{B} \cup \mbox{SET}_2(\mathcal{B})\cup \cdots =\bigcup_{k\geq 0}\mbox{SET}_k(\mathcal{B}).$$ The translation rule is $$\mathcal{A}=\mbox{SET}(\mathcal{B})\Rightarrow A(z)=\sum_{k=0}^{\infty} \frac{1}{k!}B(z)^k=e^{B(z)}.$$
We have $\mathcal{P}_S=\mbox{SET}(\mbox{SET}_S(\mathcal{X})),$ which turns into the generating function \eqref{gfunbell}.
\end{proof}

There exists a beautiful formula for the \emph{Bell numbers} (cf. \cite{Dobinski,Pitman}) called the  Dobinski's formula. It says that
$$B_n=\frac{1}{e}\sum_{k=0}^{\infty}\frac{k^n}{k!}.$$

Theorem \ref{Dobi} below generalizes the Dobinski's formula.
\begin{theorem}\label{Dobi}
Let $B_{n, S}(x)$ be the polynomial defined by
$$B_{n, S}(x)=\sum_{k=0}^n{n \brace k}_Sx^k.$$
Then
$$B_{n, S}(x)=\frac{1}{e^x}\sum_{\ell=0}^\infty f_{S, \ell}^{(n)}(0)\frac{x^\ell}{\ell!}.$$
\end{theorem}
\begin{proof}
From Equation \eqref{eqbb}  we have
\begin{align*}
\sum_{\ell=0}^\infty f_{S, \ell}^{(n)}(0)\frac{x^\ell}{\ell!}&=\sum_{\ell=0}^\infty \sum_{k=0}^n { n \brace k }_{S} (\ell)_k\frac{x^\ell}{\ell!} = \sum_{\ell=0}^\infty \sum_{k=0}^n { n \brace k }_{S}  \binom \ell k \frac{k!}{\ell!}x^\ell\\
&= \sum_{s=0}^\infty \sum_{k=0}^n { n \brace k }_{S}  \binom{s+k}{k} \frac{k!}{(s+k)!}x^{s+k}= \sum_{s=0}^\infty \frac{x^s}{s!}\sum_{k=0}^n { n \brace k }_{S}  x^{k}\\
&=e^xB_{n, S}(x).
\end{align*}
\end{proof}
We will call the polynomials $B_{S,n}(x)$ as the \emph{$S$-restricted Bell polynomials}.

\subsection{$S$-restricted Stirling Matrix}
The \emph{$S$-restricted Stirling matrix of the second kind} is the infinite matrix defined by
$$\M_S:=\left[{n \brace k}_S\right]_{n, k \geq 0}.$$

 An infinite lower triangular matrix $L=\left[d_{n,k}\right]_{n,k\in \N}$ is called an \emph{exponential  Riordan array}, (cf. \cite{Barry}), if its column $k$ has generating function $g(x)\left(f(x)\right)^k/k!, k = 0, 1, 2, \dots$, where $g(x)$ and $f(x)$ are formal power series with $g(0) \neq 0$, $f(0)=0$ and $f'(0)\neq 0$.   The matrix corresponding to the pair $f(x), g(x)$ is denoted by  $\langle g(x),f(x)\rangle$.

If we multiply  $\langle g(x),f(x)\rangle$ by a column vector $(c_0, c_1, \dots)^T$ with the exponential generating function $h(x)$ then the resulting column vector has exponential generating function $gh(f)$. This property is known as the fundamental theorem of exponential Riordan arrays or summation property.  The product of two exponential Riordan arrays $\langle g(x),f(x)\rangle$ and $\langle h(x),\ell(x)\rangle$ is defined by:
 $$\langle g(x),f(x)\rangle *\langle h(x),\ell(x)\rangle =\left\langle g(x)h\left(f(x)\right), \ell\left(f(x)\right)\right\rangle.$$
 The set of all exponential Riordan matrices is a group  under the operator $*$ (cf.\ \cite{Barry, Riordan}).\\

  For example,  the Pascal matrix $\Pa$ (cf. \cite{Riordan}) is given by the exponential Riordan matrix
 \begin{align*}
 \Pa=\langle e^x,x\rangle=\left[\binom nk \right]_{n, k \geq 0}.
 \end{align*}

From Equation \eqref{gfun} and the definition of Riordan matrix we obtain the following theorem.
\begin{theorem}
For all $S\subseteq \Z^+$ with $1\in S$, the matrix $\M_S$ is an exponential Riordan matrix given by
$$\M_S=\left\langle 1, x + \sum_{i\geq 1} \frac{x^{k_i}}{k_i!} \right\rangle,$$
with $k_i > 1$ for all positive integer $i$.
\end{theorem}
The row sum of these matrices are the $S$-restricted Bell numbers $B_{n, S}$.

From the Riordan matrix theory it is possible to obtain factorizations for the matrices $\M_S$.  Given any Riordan matrix $\langle g(x),f(x)\rangle$, it can be factorized by (cf. \cite{Cheon})
$$\langle g(x),f(x)\rangle=\langle g(x),x\rangle *\left([1]\oplus \langle f'(x),f(x)\rangle \right),$$
where $\oplus$ denotes the direct sum of two matrices.

For example,
$$\M_\Odd=\langle 1, \sinh x -1 \rangle=[1]\oplus \langle \cosh(x),\sinh(x)-1\rangle=[1]\oplus \left(\langle \cosh(x),x\rangle * \langle 1,\sinh x - 1 \rangle \right).$$

Therefore,
$$\M_{\Odd}=\prod_{\ell\geq 1}( I_\ell \oplus\overline{\Pa}),$$
where $I_\ell$ is the $\ell\times \ell$ identity matrix and $\overline{\Pa}=\langle \cosh(x),x\rangle=[c_{n,k}]_{n, k \geq 0}$ with
$$c_{n,k}=\begin{cases}
\binom nk \left(\frac{1+(-1)^{(n+k)}}{2}\right), & \text{ if }  k\geq 1; \\
\delta_{n,0},& \text{ if }  k=0;
\end{cases}$$
and $\delta_{n,k}$ is the Kronecker delta symbol.  \\

The inverse exponential Riordan array of $\M_S$ is denoted by $$\T_S:=\left[{n \brace k}_S^{-1}\right]_{n, k \geq 0}.$$

Recently, Engbers et al. \cite{Eng} gave an interesting combinatorial interpretation for the absolute  values of the entries ${n \brace k}_S^{-1}$, by means of Schr\"oder trees.

It is clear that the $S$-restricted Stirling numbers  satisfy the following orthogonality relation:
$$\sum_{i=k}^n {n \brace i}_S  {i \brace k}_S^{-1}=\sum_{i=k}^n  {n \brace i}_S^{-1} {i \brace k}_S
=\delta_{k,n}.$$

The orthogonality relations give us the inverse relations
$$f_n=\sum_{k=0}^n{n \brace k}_S^{-1} g_k  \iff g_n=\sum_{k=0}^n{n \brace k}_S f_k.$$

From definition of the polynomials $B_{n, S}(x)$ we obtain the  equality
$$X = \M_S^{-1}\mathcal{B}_S,$$
where $X=[1, x, x^2, \dots]^T$ and $\mathcal{B}_S=[B_{0,S}(x), B_{1,S}(x), B_{2,S}(x), \dots]^T$. Then  $X=\T_S\mathcal{B}_S$ and
$$x^n=\sum_{k=0}^n {n \brace k}_S^{-1} B_{k,S}(x).$$
 Therefore
 \begin{align}
 B_{n,S}(x)=x^n-\sum_{k=0}^{n-1}{n \brace k}_S^{-1} B_{k,S}(x), \quad n\geq 0.\label{iddet}
 \end{align}
 From the above identity we obtain the determinantal identity of Theorem \ref{teo3}.
 \begin{theorem}\label{teo3}
 For all $S\subseteq \Z^+$ with $1\in S$, the $S$-restricted  Bell polynomials satisfy
 $$B_{n,S}(x)=(-1)^n\begin{vmatrix}
 1 & x & & \cdots && x^{n-1} & x^n\\
  1 & {1 \brace 0}_S^{-1} & &\cdots && {n-1 \brace 0}_S^{-1} & {n \brace 0}_S^{-1}\\
    0 & 1 && \cdots && {n-1 \brace 1}_S^{-1}  & {n \brace 1}_S^{-1} \\
   \vdots &  & & \cdots& & & \vdots\\
  0 & 0  && \cdots & & 1 & {n \brace n-1}_S^{-1}\\
 \end{vmatrix}.$$
 \end{theorem}
 \begin{proof}
This identity follows from Equation \eqref{iddet} and by expanding the determinant by the last column.
 \end{proof}

 For example, if $S=\{1, 3, 6\}$ then
$$\M_{\{1, 3, 6 \}}=\left\langle 1, x + \frac{x^3}{3!} + \frac{x^6}{6!} \right\rangle=
\left(
\begin{array}{ccccccccc}
 1 & 0 & 0 & 0 & 0 & 0 & 0 & 0 & 0 \\
 0 & 1 & 0 & 0 & 0 & 0 & 0 & 0 & 0 \\
 0 & 0 & 1 & 0 & 0 & 0 & 0 & 0 & 0 \\
 0 & 1 & 0 & 1 & 0 & 0 & 0 & 0 & 0 \\
 0 & 0 & 4 & 0 & 1 & 0 & 0 & 0 & 0 \\
 0 & 0 & 0 & 10 & 0 & 1 & 0 & 0 & 0 \\
 0 & 1 & 10 & 0 & 20 & 0 & 1 & 0 & 0 \\
 0 & 0 & 7 & 70 & 0 & 35 & 0 & 1 & 0 \\
 0 & 0 & 0 & 28 & 280 & 0 & 56 & 0 & 1 \\
  \vdots  &  &  &  & \vdots &  &  &  & \vdots \\
\end{array}
\right),$$
and
$$\T_{\{1, 3, 6\}}=\left[{n \brace k}_{\{1, 3, 6\}}^{-1}\right]_{n, k \geq 0}=
\left(
\begin{array}{ccccccccc}
 1 & 0 & 0 & 0 & 0 & 0 & 0 & 0 & 0 \\
 0 & 1 & 0 & 0 & 0 & 0 & 0 & 0 & 0 \\
 0 & 0 & 1 & 0 & 0 & 0 & 0 & 0 & 0 \\
 0 & -1 & 0 & 1 & 0 & 0 & 0 & 0 & 0 \\
 0 & 0 & -4 & 0 & 1 & 0 & 0 & 0 & 0 \\
 0 & 10 & 0 & -10 & 0 & 1 & 0 & 0 & 0 \\
 0 & -1 & 70 & 0 & -20 & 0 & 1 & 0 & 0 \\
 0 & -280 & -7 & 280 & 0 & -35 & 0 & 1 & 0 \\
 0 & 84 & -2800 & -28 & 840 & 0 & -56 & 0 & 1 \\
\end{array}
\right).$$
The first few $\{1, 3, 6\}$-restricted Bell polynomials are
$$1, \, x,\, x^2,\, x^3+x,\, x^4+4x^2,\, x^5+10 x^3, \, x^6+20 x^4+10
   x^2+x, \, x^7+35 x^5+70 x^3+7 x^2, \dots$$
   In particular,
   \begin{align*}
   B_{6, \{1, 3, 6\}}(x)=x^6+20 x^4+10x^2+x =\begin{vmatrix}
 1 & x & x^2 & x^3 & x^4 & x^{5} & x^6\\
 1 & 0 & 0 & 0 & 0 & 0 & 0\\
    0 & 1 & 0& -1 & 0& 10  & -1 \\
    0 & 0 & 1& 0 & -4& 0  & 70 \\
    0 & 0 & 0&1 & 0& -10  & 0 \\
     0 & 0 & 0&0 & 1& 0  & -20 \\
     0 & 0 & 0&0 & 0& 1  & 0  \end{vmatrix}.
 \end{align*}

\section{$S$-restricted Ordered Partitions and Fubini Numbers}

The \emph{Fubini numbers}, $F_n$  (\cite{OEIS}, A000670), or also called ordered Bell numbers count the number of set partitions of $[n]$ such that the order of the blocks matters.  The Fubini numbers are given by
$$F_n=\sum_{k=0}^n k! { n \brace k}.$$
Analogously, we introduce the \emph{$S$-restricted Fubini numbers},  $F_{n,S}$, as the number of ordered set partitions of $[n]$ such that the block sizes are contained in $S$. Therefore,
\begin{align}\label{SFubini}
F_{n,S}=\sum_{k=0}^n k! { n \brace k}_S.
\end{align}
The Fubini numbers also count the number of rankings of $n$ candidates with ties allowed. It is easy to see: the $n$ candidates are first to be partitioned into equivalence classes and then the equivalence classes are to be linearly ordered. A combinatorial interpretation of the $S$-restricted Fubini number is the number of ranking $n$ candidates such that the possible number of candidates that are tied is in a given set $S$. We note here that in \cite{NM} the authors derive some connections between preferential arrangements and multi-poly-Bernoulli numbers.
\begin{theorem}\label{gfunfubini}
The exponential generating function for the $S$-restricted Fubini numbers is
\begin{align}\label{gfunSfubini}
\sum_{n=0}^\infty F_{n, S}\frac{x^n}{n!}=\frac{1}{1-E_S(x)}.
\end{align}
\end{theorem}
\begin{proof}
From  Equations  \eqref{SFubini} and \eqref{gfun} we have
\begin{align*}
\sum_{n=0}^\infty F_{n, S}\frac{x^n}{n!}&=\sum_{n=0}^\infty \sum_{k=0}^n k! { n \brace k}_S \frac{x^n}{n!}=\sum_{k=0}^\infty k!\sum_{n=k}^\infty  { n \brace k}_S \frac{x^n}{n!}\\
&=\sum_{k=0}^\infty\left(\sum_{i\geq 1} \frac{x^{k_i}}{k_i!}\right)^k = \frac{1}{1-E_S(x)}.
\end{align*}

For a combinatorial proof we can use again the symbolic method. For the construction for the class of $S$-restricted ordered partitions, denoted by $\mathcal{OP}_S$, we will need the basic construction of a sequence. A sequence formed from $\mathcal{B}$ is a union of any power of $\mathcal{B}$; where the $k$th power of $\mathcal{B}$ is the (labeled) product of $\mathcal{B}$ with $k$ factors: $$\mbox{SEQ}(\mathcal{B}):=\{\epsilon\}\cup \mathcal{B}\cup \mathcal{B}\times \mathcal{B}\cup\cdots=\bigcup_{k\geq 0}\mbox{SEQ}_k(\mathcal{B}).$$ We have
$$
\mathcal{OP}_S=\mbox{SEQ}(\mbox{SET}_{S}(\mathcal{X})).
$$
The translation rules turn this construction into the generating function \eqref{gfunSfubini}.
\end{proof}
From the theorem above we obtain the following particular cases
 \begin{align*}
 \sum_{n=0}^\infty F_{n} \frac{x^n}{n!}&=\frac{1}{2-e^x}, \\
 \sum_{n=0}^\infty F_{n,\leq m} \frac{x^n}{n!}&=\frac{1}{1-x-\frac{x^2}{2!}-\cdots-\frac{x^m}{m!}}, \\
\sum_{n=0}^\infty F_{n,\geq m} \frac{x^n}{n!}&=\frac{1}{1-\frac{x^m}{m!}-\frac{x^{m+1}}{(m+1)!}-\cdots},\\
\sum_{n=0}^\infty F_{n,\E} \frac{x^n}{n!}&=\frac{1}{2-\cosh x},\\
\sum_{n=0}^\infty F_{n,\Odd} \frac{x^n}{n!}&=\frac{1}{1-\sinh x}.
 \end{align*}

Some determinantal identities for the incomplete Fubini numbers $ F_{n,\leq m}$ and $F_{n,\geq m}$ were given in \cite{KomRam}.

 For the Fubini numbers there is a formula similar to Dobinski's formula. It says that (cf. \cite{Pip})
 $$F_n=\frac{1}{2}\sum_{k=0}^\infty \frac{k^n}{2^k}.$$

In the following theorem we give a generalization of this equation.

\begin{theorem}\label{teo5}
We have the combinatorial identity
$$F_{n, S}=\frac 12 \sum_{k=0}^\infty \frac{1}{2^k}\sum_{\ell=0}^n {n \brace \ell}_S(k)_\ell.$$
\end{theorem}
\begin{proof}
From Theorem \ref{gfunfubini} we have
\begin{align*}
\sum_{n=0}^\infty F_{n, S}\frac{x^n}{n!}&=\frac{1}{1-E_S(x)}=\frac{1}{2\left(1-\frac 12 \left(1+E_S(x)\right)\right)}\\
&=\frac{1}{2} \sum_{k=0}^\infty \frac{1}{2^k} \left(1+E_S(x)\right)^k=\frac{1}{2} \sum_{k=0}^\infty \frac{1}{2^k}f_{S,k}(x).
\end{align*}
Since
$$[x^n]f_{S,k}(x)=\frac{1}{n!}\left.\frac{d^n}{dx^n}f_{S,k}(x)\right |_{x=0}=\frac{1}{n!}\sum_{\ell=0}^n{n \brace \ell}_{S} (k)_\ell,
$$
then
\begin{align*}
\sum_{n=0}^\infty F_{n, S}\frac{x^n}{n!}&=\frac{1}{2} \sum_{k=0}^\infty \frac{1}{2^k} \sum_{n=0}^\infty\left(\frac{1}{n!}\sum_{\ell=0}^n{n \brace \ell}_{S} (k)_\ell \right)x^n= \sum_{n=0}^\infty \left(\frac{1}{2} \sum_{k=0}^\infty \frac{1}{2^k}\sum_{\ell=0}^n{n \brace \ell}_{S} (k)_\ell \right)\frac{x^n}{n!}.
\end{align*}
By comparing the $n$-th coefficients we obtain the desired result.
\end{proof}

 Some other identities involving Fubini numbers occur in the literature. For example,  Poonen \cite{Poonen} proved by induction that
  \begin{align*}
(2^q-1)F_n&=\sum_{\ell=0}^{n-1}q^{n-\ell}\binom n \ell F_\ell + \sum_{\ell=1}^{q-1}2^{q-\ell-1}\ell^n, \quad n, q \geq 0.
\end{align*}
Recently, Diagana and Ma\"iga \cite{DM} gave a proof by means of  Laplace transform associated with a $p$-adic measure. This  identity can be generalized for the $S$-restricted Fubini numbers.

\begin{theorem}
Let $n$ and $q$ be positive integers. Then the $S$-restricted Fubini numbers satisfy the identity
\begin{align*}
2^qF_{n,S}=\sum_{\ell=0}^n\binom n\ell F_{\ell, S}\left(\sum_{i=0}^{n-\ell}{n-\ell \brace i }_{S}(q)_i\right) + \sum_{\ell=1}^{q-1}2^{q-\ell-1}\left(\sum_{i=0}^{n}{n\brace i }_{S}(\ell)_i\right).
\end{align*}
\end{theorem}
\begin{proof}
From Theorem \ref{teo5} and Equation \eqref{eqbb} we  have
\begin{align*}
F_{n,S}&=\frac12\sum_{k=0}^{q-1} \frac{1}{2^k}\sum_{\ell=0}^n{n \brace \ell}_{S}(k)_\ell + \frac12\sum_{k=q}^\infty \frac{1}{2^k}\sum_{\ell=0}^n{n \brace \ell}_{S}(k)_\ell\\
&=\frac12\sum_{k=0}^{q-1} \frac{1}{2^k}\sum_{\ell=0}^n{n \brace \ell}_{S}(k)_\ell + \frac{1}{2^{q+1}}\sum_{k=0}^\infty \frac{1}{2^k}\sum_{\ell=0}^n{n \brace \ell}_{S}(k+q)_\ell\\
&=\frac12\sum_{k=0}^{q-1} \frac{1}{2^k}\sum_{\ell=0}^n{n \brace \ell}_S (k)_\ell + \frac{1}{2^{q+1}}\sum_{k=0}^\infty\frac{1}{2^k} f_{S,k+q}^{(n)}(0)\\
&=\frac12\sum_{k=0}^{q-1} \frac{1}{2^k}\sum_{\ell=0}^n{n \brace \ell}_{S}(k)_\ell + \frac{1}{2^{q+1}}\sum_{k=0}^\infty\frac{1}{2^k} \sum_{j=0}^n\binom nj f_{S,k}^{(j)}(0) f_{S,q}^{(n-j)}(0)\\
&=\frac12\sum_{k=0}^{q-1} \frac{1}{2^k}\sum_{\ell=0}^n{n \brace \ell}_{S}(k)_\ell + \frac{1}{2^{q}}\sum_{j=0}^n\binom nj \left( \frac 12\sum_{k=0}^\infty \frac{1}{2^k} f_{S,k}^{(j)}(0)\right) f_{S,q}^{(n-j)}(0)\\
&=\frac12\sum_{k=0}^{q-1} \frac{1}{2^k}\sum_{\ell=0}^n{n \brace \ell}_{S}(k)_\ell + \frac{1}{2^{q}}\sum_{j=0}^n\binom nj F_{j,S} \left(\sum_{i=0}^{n}{n-j\brace i }_{S}(q)_i\right).
\end{align*}
\end{proof}

From the theorem above we obtain the following congruence
\begin{align*}
(2^q-1)F_{n,S}\equiv\sum_{\ell=1}^{q-1}2^{q-\ell-1}\left(\sum_{i=0}^{n}{n\brace i }_{S}(\ell)_i \right) \pmod q.\end{align*}

\section{$S$-restricted Lonesum Matrices}

\subsection{The structure of a lonesum matrix}

\emph{Lonesum matrices} are by definition $(0,1)$ matrices that are uniquely reconstructible from their column and row sum vectors. The investigation of lonesum matrices started at the birth of discrete tomography, the research of the study of reconstruction of objects from partial informations about them, as from their projections. Lonesum matrices are in bijection with several other combinatorial objects, as for instance the natural decoding of an acyclic orientation of a complete bipartite graph $K_{n,k}$ is a $n\times k$ lonesum matrix. In \cite{BD,BH1,BH2} we find further corresponding combinatorial objects and bijections.

Ryser \cite{Ryser} described some basic properties of these matrices. Lonesum matrices avoid both $2\times 2$ permutation matrices. (A matrix $A$ \emph{avoids} a matrix $B$ if none of the submatrix of $A$ is equal to $B$.) It was also shown that two columns (resp. rows) of a lonesum matrix with the same column sum (resp. row sum) is the same (the positions of the 1 entries are the same). Moreover, for two columns $C_1$ and $C_2$ with column sums $|C_1|<|C_2|$ it is true that in the positions of the 1 entries of the column $C_1$ there are also 1 entries in the column $C_2$. (The same is true for rows.) Hence, if we order the different columns (resp. rows) with at least one 1 entry of a lonesum matrix according their sum: $\{C_1,C_2,\ldots, C_m\}$, two consecutive columns $C_i$ and $C_{i+1}$ differs only in the positions, where $C_i$ has $0$s while $C_{i+1}$ has $1$s. So, the different columns are uniquely determined by an ordered partition: $OP=\{B_1,\ldots, B_m\}$, where $B_1$ contains the positions of 1s in $C_1$ and for $i>1$ $B_i$ consists of the positions of 1s in $C_{i+1}$ in that $C_i$ has $0$s. (The same argumentation works for rows.) It is easy to see, that the number of different columns and rows with at least one 1 entry has to be equal. 

We conclude that the structure of the lonesum matrices can be described as follows: it contains some all-zero columns and all-zero rows; and the remainder of the matrix can be encoded by an ordered partition of the row indices and an ordered partition of the column indices, where both partitions have the same number of blocks.

\begin{example}
Let $n=k=8$. Assume that there is no all-zero column and all zero rows. Let  $P_R=\{1,8\},\{5,6\},\{2,7\},\{3,4\}$ be an ordered partition of $[n]$ and \linebreak $P_C$=$\{2,5\},\{4,7\},\{1,3\},\{6,8\}$ be an ordered partition of $[k]$. The lonesum matrix corresponding to $(P_R, P_C)$ is $M$:
\begin{align*}M=
\begin{pmatrix}
0&1&0&0&1&0&0&0\\
1&1&1&1&1&0&1&0\\
1&1&1&1&1&1&1&1\\
1&1&1&1&1&1&1&1\\
0&1&0&1&1&0&1&0\\
0&1&0&1&1&0&1&0\\
1&1&1&1&1&0&1&0\\
0&1&0&0&1&0&0&0
\end{pmatrix}
\end{align*}
\end{example}

Let $\Lo(n,k)$ denote the set of lonesum matrices with $n$ rows and $k$ columns. To keep our formula simple, we introduce a ``dummy'' element to both sets $[n]$ and $[k]$; and  we partition these modified sets into $(m+1)$ non-empty blocks. The row indices and column indices that share the same block with the dummy elements, will be the all-zero rows and all zero columns. The remaining  $m$ ``ordinary'' blocks of $[n]$ (as well as of $[k]$) have to be ordered. Hence, we have \cite{Brew}:
$$
|\Lo(n,k)|=\sum_{m=0}^{\min(n,k)}m!{ {n+1} \brace {m+1} }m!{ {k+1} \brace {m+1}}.
$$
These numbers are the poly-Bernoulli numbers with negative $k$ indices. Poly-Bernoulli numbers and their generalizations and analogues have a rich literature. In a later section we consider the generalization/restriction of poly-Bernoulli numbers itself from the analytic point of view using another well-known formula for them.

\subsection{$S$-restricted lonesum matrices}
Next we consider lonesum matrices with restriction on the number of the columns and rows of the same type. \emph{The type} of two columns is the same if the columns coincide in each position. The first author of this paper investigated the case when the number of the columns and rows is at most $d$ in \cite{Benyi}. In this paper we study a more general case.

We call a lonesum matrix \emph{$(S_1,S_2)$-restricted} if the number of columns of the same type are contained in the set $S_1=\{k_1, k_2, \dots \}$, and the number of rows of the same type are contained in $S_2=\{s_1, s_2, \dots \}$. We let $\Lo_{S_1,S_2}$ denote the set of $(S_1,S_2)$-restricted lonesum matrices without all-zero columns and rows. Further, we let $\Lo_{S_1,S_2}(n,k)$ denote the set of $(S_1,S_2)$-restricted lonesum matrices with $n$ rows and $k$ columns without all-zero rows and columns.

\begin{theorem} We have
$$|\Lo_{S_1,S_2}(n,k)|=\sum_{m=0}^{\min(n,k)}m! {n \brace m}_{S_1} m! {k \brace m}_{S_2}.$$
\end{theorem}
\begin{proof}
From the characterization described above it follows that a pair of ordered partitions determines a lonesum matrix. Let consider this bijection in more detail. The partition of the row indices determines the columns. It is also true, that if two row indices $r_1$ and $r_2$ are in the same block, then the rows $R_1$ and $R_2$ have to be of the same type. Hence, restrictions on the sizes of the blocks of the partition of the set of row indices correspond to restrictions on the number of rows of the same type in the lonesum matrix. Similarly, restrictions on the sizes of the blocks of the partition of the set of column indices correspond to restrictions on the number of columns of the same type. The statement follows.
\end{proof}

The set $\Lo_{S_1,S_2}(n,k)$ also may be empty. It depends on the relation between the numbers: $n$, $k$ and the elements of $S_1$ and $S_2$ whether a lonesum matrix with such a restriction exists or not. For instance, the set $\Lo_{S_1,S_2}(n,k)$ is empty if $S_1=S_2=\{1,2,\ldots, d\}$ and $k<nd$.
Let $L_{S_1,S_2}(x,y)$ denote the double exponential generating function of $(S_1,S_2)$-restricted lonesum matrices without all-zero
columns and rows.
$$
L_{S_1,S_2}(x,y)=\sum_{M\in \Lo_{S_1,S_2}}\frac{x^{r(M)}}{r(M)!}\frac{y^{c(M)}}{c(M)!}.
$$
\begin{theorem}We have
\begin{align}\label{gfunlonesum}
L_{S_1,S_2}(x,y)=\frac{1}{1-E_{S_1}(x)E_{S_2}(y)}.
\end{align}
\end{theorem}
\begin{proof}
The well understood structure of lonesum matrices allows us to use again the symbolic method. We have now two atomic classes, $\mathcal{X}$ and $\mathcal{Y}$ with exponential generating function $x$ and $y$. We will need here the notion of labeled product. The labeled product of object classes $\mathcal{B}$ and $\mathcal{C}$ is obtained by forming pairs $(\beta,\gamma)$ with $\beta\in \mathcal{B}$ and $\gamma\in \mathcal{C}$, performing all possible order-consistent relabelings. The translation rule is simple $\mathcal{A}=\mathcal{B}\times\mathcal{C}\Rightarrow A(x)=B(x)C(x)$.
An ordered pair of partitions of the sets $[k]$ and $[n]$ is a sequence of pairs of sets formed from elements in $[k]$ resp. $[n]$, with the restriction on the size of the blocks/sets. Hence, the construction of lonesum matrices in $\Lo_{S_1,S_2}$ (we use the same notation for the object class) is:
\begin{align*}
\Lo_{S_1,S_2}=\mbox{SEQ}\left(\mbox{SET}_{S_1}(\mathcal{X})\times \mbox{SET}_{S_2}(\mathcal{Y})\right).
\end{align*}
Note that we need the same number of sets of $[n]$ and $[k]$, it is not allowed in a pair of a set of $[n]$ and $[k]$ that one of the set is empty.
The translation rules of the symbolic method turn this construction into the generating function \eqref{gfunlonesum}.
\end{proof}
The next theorem is straightforward.
\begin{theorem}
Let $\widetilde{L}_{S_1,S_2}(x,y)$ be the generating function of $(S_1,S_2)$-restricted lonesum matrices. Then we have
\begin{align}\label{gfunlonesum2}
\widetilde{L}_{S_1,S_2}(x,y)=\frac{1+E_{S_1}(x)+E_{S_2}(y)+E_{S_1}(x)E_{S_2}(y)}{1-E_{S_1}(x)E_{S_2}(y)}.
\end{align}
\end{theorem}
\begin{proof}
A lonesum matrix is built up of a (possible empty) set of all-zero columns, a (possible empty) set of all-zero rows and a lonesum matrix without all-zero columns and rows. The number of all-zero rows (resp. columns) has to be contained in the set $S_1$ (resp. $S_2$ ). We will here need the empty class $\epsilon$ with generating function $1$ for taking care also about the possible empty sets. Our construction is:
$$
\widetilde{\Lo}_{S_1,S_2}=(\mbox{SET}_{S_1}(\mathcal{X})+\epsilon)\times(\mbox{SET}_{S_2}(\mathcal{Y})+\epsilon)\times\mbox{SEQ}\left(\mbox{SET}_{S_1}(\mathcal{X})\times \mbox{SET}_{S_2}(\mathcal{Y})\right).
$$
This turns automatically into \eqref{gfunlonesum2}.
\end{proof}

For example, the bivariate generating function for the classical lonesum matrices is
$$   \widetilde{L}_{\Z,\Z}= \frac{e^{x+y}}{e^x+e^y-e^{x+y}}  $$
 The bivariate generating function for the lonesum matrices such that the numbers of columns (resp. rows) of the same type are contained in the set $\mathcal{E}$ (resp. $\mathcal{O}$) is
 $$   \widetilde{L}_{\mathcal{E},\mathcal{O}}= \frac{\cosh x + \cosh x \sinh y}{1-\cosh x\sinh y + \sinh y }.$$

\subsection{$S$-restricted lonesum decomposable matrices}
In this subsection we derive some results on lonesum decomposable matrices with restriction on the number of columns and rows of the same type. Lonesum decomposable matrices were recently introduced by Kamano in \cite{Kamano}.

A $(0,1)$ matrix $D$ is \emph{lonesum decomposable} if by permuting its rows and/or columns it can be written in a form:
$$
\begin{pmatrix}
L_1&&&\\
 &L_2&&\\
&&\ddots &\\
&&&L_r
\end{pmatrix},
$$
where $L_i$ $(1\leq i\leq r)$ are lonesum matrices. Kamano showed that for each lonesum decomposable matrix $r$ is uniquely determined and called $r$ the \emph{decomposition order}. Kamano proved that a matrix is lonesum decomposable if and only if it does not contain any submatrix that can be obtained from $U$ and $U^t$ by permuting rows and columns, where
$$
U=\begin{pmatrix}
1&1&0\\
1&0&1
\end{pmatrix}.
$$

It is interesting that a similar characterization was also given for the lonesum matrices with the matrix $\left(\begin{array}{cc} 1&0\\0&1\end{array}\right)$. $01$ matrices are also investigated in extremal combinatorics. The main problem in this setting is the determination of the maximal number of $1$'s in a $n\times m$ $01$ matrix not having a certain matrix $M$ (or set of matrices) as a submatrix. This number is denoted usually by $\mbox{ext}(n,m;M)$. The matrix $U$ is a special matrix from this point of view, since it is the smallest matrix for that $\mbox{ext}(n,n;U)=\Theta(n\log n)$ \cite{FH}. (The notation $f(n)=\Theta(g(n))$ means that the function $f(n)$ is asymptotically bounded both above and below by the function $g(n)$, i.e. there are positive numbers $k_1,k_2$ such that there is an index $n_0$ such that for all $n>n_0$ $k_1g(n)\leq f(n)\leq k_2g(n)$ holds.)

We mention here that $U$ is one of the smallest matrix with $\mbox{ext}(n; U)=\Theta(n\log n)$. $\mbox{ext}(n; U)$ is the maximal number of $1$ entries that a $n\times n$ matrix that avoid $U$ as a submatrix can have.

Let $D_{r,S_1,S_2}(n,k)$ denote the number of $n\times k$ lonesum decomposable matrices of decomposition order $r$ such that the number of rows of the same type is contained in $S_1$ and the number of columns of the same type is contained in $S_2$. Let $D_{S_1,S_2}$ denote the number of lonesum decomposable matrices such that the number of rows of the same type is contained in $S_1$ and the number of columns of the same type is contained in $S_2$.
\begin{theorem}For $D_{r,S_1,S_2}(n,k)$ and $D_{S_1,S_2}(n,k)$ we have the following generating functions:
$$
\sum_{n=0}^{\infty} \sum_{k=0}^{\infty} D_{r,S_1,S_2}(n,k)\frac{x^n}{n!}\frac{y^k}{k!}= \frac{(1+E_{S_1}(x))(1+E_{S_2}(y))}{r!}
\left(\frac{1}{1-E_{S_1}(x)E_{S_2}(y)}-1\right)^r,
$$
and
$$
\sum_{n=0}^{\infty} \sum_{k=0}^{\infty} D_{S_1,S_2}(n,k)\frac{x^n}{n!}\frac{y^k}{k!}=(1+E_{S_1}(x))(1+E_{S_2}(y))\times
\exp\left(\frac{1}{1-E_{S_1}(x)E_{S_2}(y)}-1\right).
$$
\end{theorem}
\begin{proof}
A lonesum decomposable matrix with decomposition order $r$ is a (not empty) $r$-set of lonesum matrices without all-zero rows and columns, a (possible empty) set of all-zero columns and a (possible empty) set of all zero-rows.

Using the construction of lonesum matrices and taking care of the restriction on the size of the sets (the number of the same type of columns and rows) we have the construction:
$$
\mathcal{D}_{r;S_1,S_2}=(\mbox{SET}_{S_1}(\mathcal{X})+\epsilon)\times (\mbox{SET}_{S_2}(\mathcal{Y})+\epsilon)\times \mbox{SET}_r(\mbox{SEQ}_{>0}\left(\mbox{SET}_{S_1}(\mathcal{X})\times \mbox{SET}_{S_2}(\mathcal{Y})\right)).
$$
For lonesum decomposable matrices we have a set of lonesum matrices (without restriction on the size) intead of $r$-set of lonesum matrices.
$$
\mathcal{D}_{S_1,S_2}=(\mbox{SET}_{S_1}(\mathcal{X})+\epsilon)\times (\mbox{SET}_{S_2}(\mathcal{Y})+\epsilon)\times \mbox{SET}(\mbox{SEQ}_{>0}\left(\mbox{SET}_{S_1}(\mathcal{X})\times \mbox{SET}_{S_2}(\mathcal{Y})\right)).
$$
The translation rules of the symbolic method turn the construction into the formulas given in the theorem.
\end{proof}

For example, the bivariate generating function for the classical lonesum decomposable matrices are
$$
\sum_{n=0}^{\infty} \sum_{k=0}^{\infty} D_{r,\Z,\Z}(n,k)\frac{x^n}{n!}\frac{y^k}{k!}= \frac{e^{x+y}}{r!}
\left(\frac{1}{e^x+e^y-e^{x+y}}-1\right)^r,
$$
and
$$
\sum_{n=0}^{\infty} \sum_{k=0}^{\infty} D_{\Z,\Z}(n,k)\frac{x^n}{n!}\frac{y^k}{k!}=
\exp\left(\frac{1}{e^x+e^y-e^{x+y}}+x+y-1\right).
$$

\section{$S$-restricted poly-Bernoulli Numbers}

The \emph{poly-Bernoulli numbers}  $B_n^{(k)}$ were introduced by Kaneko  \cite{Kaneko} by using the following generating function
\begin{align*}
\frac{\li_k(1-e^{-t})}{1-e^{-t}}=\sum_{n=0}^{\infty}B_n^{(k)}\frac{t^n}{n!}, \quad (k \in \Z)
\end{align*}
where
\begin{align*}
\li_k(t)=\sum_{n=1}^\infty\frac{t^n}{n^k}
\end{align*}
is the $k$-th \emph{polylogarithm function}. If $k=1$ we get $B_n^{(1)}=(-1)^nB_n$ for $n\geq 0$, where $B_n$ are the Bernoulli numbers.

The poly-Bernoulli numbers $B_n^{(k)}$ \cite[Theorem 1]{Kaneko} can be defined by means of Stirling numbers of the second kind as follows:
\begin{align}
B_n^{(k)}=\sum_{i=0}^n{n \brace i}\frac{(-1)^{n-i}i!}{(i+1)^k}\,. \label{polber}
\end{align}

Brewbaker \cite{Brew} proved combinatorially that for any positive integer $n$ and $k$,
$$|\Lo(n,k)|=B_{n}^{(-k)}.$$
In the proof \cite{Brew} Brewbaker considered a $01$ matrix as a sequence of its columns. Given $i$ different not all-zero columns, there are $i^k$ possibilities to build up a matrix. However, since in a lonesum matrix the permutation matrices of size $2\times 2$ can not appear as a submatrix, there are restrictions on a set of columns from that we can built up a lonesum matrix. The cardinality of such a set is ${n\brace i}i!$. Hence, there are ${n\brace i}i!(i+1)^k$  lonesum matrices with at most $i$ different columns. An inclusion-exclusion argument finishes the proof of the formula \eqref{polber} with negative $k$ values.

We define the  \emph{$S$-restricted poly-Bernoulli numbers} by the expression 
\begin{align}\label{defipolyB}
B_{n,S}^{(k)}=\sum_{i=0}^n{n \brace i}_{S}\frac{(-1)^{n-i}i!}{(i+1)^k}.
\end{align}

This formula can not be generalized automatically to $S$-restricted lonesum matrices. In the argumentation cited above the inclusion-exclusion process fails in this general case. The numbers $B_{n,S}^{(-k)}$ are not even positive integers for any $n$, $k$ and $S$. However, poly-Bernoulli numbers are important for their own sake; hence, we derive the generating function and an interesting property for $S$-restricted poly-Bernoulli numbers. 

In Theorem \ref{teo1polyB} we give the generating functions of $S$-restricted poly-Bernoulli numbers  in terms of $E_S(t)$.

\begin{theorem}\label{teo1polyB}
The exponential generating function of the $S$-restricted poly-Bernoulli numbers is
$$\sum_{n=0}^\infty  B_{n,S}^{(k)}\frac{t^n}{n!}= \frac{{\rm Li}_k\bigl(-E_S(-t) \bigr)}{-E_S(-t)}.$$
\end{theorem}
\begin{proof}
By Definition \eqref{defipolyB} and using (\ref{gfun}), we have
\begin{align*}
\sum_{n=0}^\infty  B_{n,S}^{(k)}\frac{t^n}{n!}&=\sum_{n=0}^\infty\sum_{\ell=0}^n{n \brace \ell}_{S}\frac{(-1)^{n-\ell}\ell!}{(\ell+1)^k}\frac{t^n}{n!}=\sum_{\ell=0}^\infty\frac{(-1)^\ell \ell!}{(\ell+1)^k}\sum_{n=\ell}^\infty{n \brace \ell}_{S}\frac{(-t)^n}{n!}\\
&=\sum_{\ell=0}^\infty\frac{(-1)^\ell \ell!}{(\ell+1)^k}\frac{1}{\ell!} \left(\sum_{i\geq 1} \frac{(-t)^{k_i}}{k_i!}\right)^\ell=\sum_{\ell=0}^\infty\frac{(-E_S(-t))^\ell}{(\ell+1)^k}=\frac{{\rm Li}_k\bigl(-E_S(-t) \bigr)}{-E_S(-t)}\,.
\end{align*}
\end{proof}

The following corollary follows from some particular cases of $S$.

\begin{corollary}
We have
\begin{align*}
\sum_{n=0}^\infty B_{n}^{(k)}\frac{t^n}{n!}&=\frac{{\rm Li}_k\bigl(1-e^{-t}\bigr)}{1-e^{-t}}\,,\\
\sum_{n=0}^\infty B_{n,\leq m}^{(k)}\frac{t^n}{n!}&=\frac{{\rm Li}_k\bigl(1-E_{m}(-t)\bigr)}{1-E_{m}(-t)}\,,
\\
\sum_{n=0}^\infty B_{n,\geq m}^{(k)}\frac{t^n}{n!}&=\frac{{\rm Li}_k\bigl(E_{m-1}(-t)-e^{-t}\bigr)}{E_{m-1}(-t)-e^{-t}}\,.
\\
\sum_{n=0}^\infty B_{n,\E}^{(k)}\frac{t^n}{n!}&=\frac{{\rm Li}_k\bigl(1-\cosh(t)\bigr)}{1-\cosh t}\,.\\
\sum_{n=0}^\infty B_{n,\Odd}^{(k)}\frac{t^n}{n!}&=\frac{{\rm Li}_k\bigl(\sinh t \bigr)}{\sinh t}\,.
\end{align*}
\end{corollary}
\noindent

From the relation
\begin{align}\label{deri}
\li_k'(x)=\frac{1}{x}\li_{k-1}(x)
\end{align}
we obtain the following identity involving iterated integrals.
\begin{theorem}\label{teo1polyC}
The exponential generating functions of the $S$-restricted poly-Bernoulli numbers satisfies
\begin{multline*}
\sum_{n=0}^\infty  B_{n,S}^{(k)}\frac{x^n}{n!}=\frac{1}{-E_S(-x)}\underbrace{\int_{0}^x\frac{-E_S'(-t)}{E_S(-t)}\int_{0}^x\frac{-E_S'(-t)}{E_S(-t)}\cdots \int_{0}^x\frac{-E_S'(-t)}{E_S(-t)}}_{k-1} \\\times (-\log(1+E_S(-t))) \underbrace{dt\cdots dt}_{k-1}
\end{multline*}
\end{theorem}
\begin{proof}
From Equation \eqref{deri} we have for $k\geq 1$
$$\frac{d}{dx}\li_k(-E_S(-x))=\frac{-E_S'(-x)}{E_S(-x)}\li_{k-1}(-E_S(-x)).$$
Therefore
\begin{align*}
&\li_k(-E_S(-x))=\int_{0}^x\frac{-E_S'(-t)}{E_S(-t)}\li_{k-1}(-E_S(-t))dt\\
&=\int_{0}^x\frac{-E_S'(-t)}{E_S(-t)}\int_{0}^x\frac{-E_S'(-t)}{E_S(-t)} \li_{k-2}(-E_S(-t))dtdt\\
&=\int_{0}^x\frac{-E_S'(-t)}{E_S(-t)}\int_{0}^x\frac{-E_S'(-t)}{E_S(-t)} \cdots \int_{0}^x\frac{-E_S'(-t)}{E_S(-t)} \li_{1}(-E_S(-t))\underbrace{dt\cdots dt}_{k-1}\\
&=\int_{0}^x\frac{-E_S'(-t)}{E_S(-t)}\int_{0}^x\frac{-E_S'(-t)}{E_S(-t)} \cdots \int_{0}^x\frac{-E_S'(-t)}{E_S(-t)}(-\log(1+E_S(-t))) \underbrace{dt\cdots dt}_{k-1}.
\end{align*}
\end{proof}

 \subsection*{Acknowledgements}
 The authors would like to thank the anonymous referee for  some  useful comments. The research of Jos\'e L. Ram\'irez was partially supported by Universidad Nacional de Colombia, Project No. 37805.

\end{document}